\documentclass[12pt]{amsart}
\usepackage[utf8]{inputenc}
\usepackage{amssymb,amsthm,enumerate}
\usepackage{eucal,mathrsfs,esvect}
\usepackage[bbgreekl]{mathbbol}
\usepackage{color}
\usepackage[all]{xy}
\usepackage{fullpage}
\usepackage{mathabx}
\usepackage{hyperref}
\DeclareMathAlphabet{\mathcalligra}{T1}{calligra}{m}{n}
\usepackage[normalem]{ulem}
\usepackage[cal=boondoxo, scr=boondoxo]{mathalfa}
\usepackage{calc}
\newtheorem{lem}{Lemma}[subsection]
\newtheorem{thm}[lem]{Theorem}
\newtheorem*{thm*}{Theorem}
\newtheorem*{metathm}{Meta-Theorem}

\newtheorem{cor}[lem]{Corollary}

\theoremstyle{definition}
\newtheorem{defn}[lem]{Definition}
\newtheorem{rem}[lem]{Remark}
\newtheorem{notn}[lem]{Notation}

\usepackage{baskervald}
\newcommand{\simto}{\xrightarrow{\sim}}

\makeatletter
\newcommand*\@dblLabelI {}
\newcommand*\@dblLabelII {}
\newcommand*\@dblequationAux {}

\def\@dblequationAux #1,#2,%
    {\def\@dblLabelI{\label{#1}}\def\@dblLabelII{\label{#2}}}

\newcommand*{\doubleequation}[3][]{%
    \par\vskip\abovedisplayskip\noindent
    \if\relax\detokenize{#1}\relax
       \let\@dblLabelI\@empty
       \let\@dblLabelII\@empty
    \else 
       \@dblequationAux #1,%
    \fi
    \makebox[0.4\linewidth-1.5em]{%
     \hspace{\stretch2}%
     \makebox[0pt]{$\displaystyle #2$}%
     \hspace{\stretch1}%
    }%
    \makebox[0.2\linewidth]{\ }
    \makebox[0.4\linewidth-1.5em]{%
     \hspace{\stretch1}%
     \makebox[0pt]{$\displaystyle #3$}%
     \hspace{\stretch2}%
    }%
    \makebox[3em][r]{(%
  \refstepcounter{equation}\theequation\@dblLabelI,
  \refstepcounter{equation}\theequation\@dblLabelII)}%
  \par\vskip\belowdisplayskip
}
\makeatother

\renewcommand{\mathbb}{\mathbf}

\newcommand{\mc}[1]{\mathcal #1}
\newcommand{\ms}[1]{\mathscr #1}

\newcommand{\oper}[1]{\operatorname{#1}}

\newcommand{\coker}{\oper{coker}}

\newcommand{\Hom}{\oper{Hom}}

\newcommand{\Spec}{\oper{Spec}}

\newcommand{\ob}{\oper{ob}}
\newcommand{\eq}{\oper{Eq}}

\newcommand{\til}{\widetilde}
\newcommand{\wh}{\widehat}

\DeclareMathOperator{\id}{id}
\DeclareMathOperator{\rHom}{Fun_{r-Ex,\otimes}}

\DeclareMathOperator{\cB}{\mathcal{B}}
\DeclareMathOperator{\cA}{\mathcal{A}}

\DeclareMathOperator{\cG}{\mathcal{G}}

\DeclareMathOperator{\cX}{\mathcal{X}}
\DeclareMathOperator{\cS}{\mathcal{S}}

\DeclareMathOperator{\Ext}{\operatorname{Ext}}

\DeclareMathOperator{\Coh}{\operatorname{Coh}}

\usepackage[OT2,T1]{fontenc}
\DeclareSymbolFont{cyrletters}{OT2}{wncyr}{m}{n}
\DeclareMathSymbol{\Sha}{\mathalpha}{cyrletters}{"58}

\def\<{\left<}
\def\>{\right>}

\newcommand{\thing}[2]{
  {\xymatrix{#1 \ar@<3pt>[r] \ar@<-3pt>[r] & #2 }}}
\newcommand{\things}[4]{
  {\xymatrix{#1 \ar@<3pt>[r]^{#3} \ar@<-3pt>[r]_{#4} & #2 }}}

\title[Tannakian patching]{A Tannakian approach to patching}
\author{Bastian Haase}
\author{Daniel Krashen}
\author{Max Lieblich}

\date{\today}

\begin{document}

\begin{abstract}
  We use Tannakian methods to show that patching for coherent sheaves implies patching for objects in any Noetherian algebraic stack with affine stabilizers. Among other things, this gives a straightforward way to prove patching for torsors under linear algebraic groups, as well as patching for sheaves and torsors on proper algebraic spaces.
\end{abstract}

\maketitle

\tableofcontents

\newcommand{\stab}[3]{{\oper{Stab}_{#1}^{#2}(#3)}}
\newcommand{\st}[1]{\stab{p}{}{#1}}
\newcommand{\vbun}[2]{\oper{VB}^{{#1}}({{#2}})}
\newcommand{\vb}[1]{\vbun{}{#1}}
\newcommand{\qc}[1]{\oper{QCoh}{#1}}
\newcommand{\coh}[1]{\oper{Coh}{#1}}

\newcommand{\arr}[2]{
\xymatrix{
\ar@<-.5ex>[r]_{#1} \ar@<.5ex>[r]^{#2} &
}
}
\newcommand{\arrr}[5]{
\xymatrix{
#1 \ar[r] & #2 \ar@<-.5ex>[r]_{#4} \ar@<.5ex>[r]^{#5} & #3
}
}
\section{Introduction}

In this paper, we use Tannakian methods to show that the following holds in all
typical circumstances.

\begin{metathm}
  If patching holds for coherent sheaves then it holds for objects in any
  Noetherian algebraic stack with affine stabilizers.
\end{metathm}

Thus, for example, patching holds for principal bundles for linear algebraic
groups over proper schemes. As a consequence, we obtain new contexts to apply
methods of field patching, as well as new Meyer-Vietoris type sequences for
étale cohomology groups. This holds for various types of patching contexts
(formal, rigid, field patching, etc.).

The technique of field patching has been exploited in various ways to obtain
information about algebraic structures (particularly torsors for linear
algebraic groups) over function fields of curves over complete discretely
valued fields. The question of finding similar techniques to handle function
fields of higher dimensional varieties is still quite open, and of much
interest. The present manuscript provides a relative context in which we can now
obtain results in this direction.

The methods of this paper break up naturally into two parts. In the first part,
we explain how the Tannakian formalism gives a general context in which one can
extend patching results for coherent sheaves to similar results for objects in
any Noetherian algebraic stack with affine stabilizers, including $G$-torsors
for $G$ a linear algebraic group (Theorem~\ref{tannaka main}).\footnote{It is
interesting to note that this is somewhat implicit in \cite[lines~22-23,
page~358]{Faltings:Verlinde}.} This can then by applied to obtain new cases of
patching for torsors over rings and fields (see Corollaries~\ref{formal
patch},~\ref{thickened patch} and~\ref{pop cor}).

In the second part, we show that patching for coherent sheaves over a base
scheme $X$ can often be extended to patching for coherent sheaves
relative to a proper morphism $S \to X$. Combined with the above
results, we then obtain patching results for $G$-torsors on proper
algebraic spaces over $X$ (see Theorem~\ref{relative patch}).

\section{Notation and preliminaries}

\subsection{Generalities on 2-equalizers}

In this paper, by 2-category, we will mean what is often referred to as a strict
2-category, and will follow the definitions of, for example,
\cite[Section~I.2.1]{Gray}. The concept of a (strict) 2-equalizer can both 
refer to a
specific construction in the 2-category of categories, or an object with
universal properties in a general 2-category. We will have use for both and
define them below:
\begin{defn}[2-equalizers of morphisms in 2-categories]
  Let $\ms C$ be a 2-category, and $f_0, f_1: A_0 \to A_1$ a pair of morphisms
  in $\ms C$. A \emph{2-equalizer\/} of the morphisms $f_0, f_1$ is
  \begin{enumerate}
  \item an object $A$,
  \item a $1$-morphism $f:A\to A_0$, and
  \item a 2-isomorphism $\alpha: f_1 f\implies f_0 f$
  \end{enumerate}
  satisfying the following properties: 
  \begin{enumerate}
  \item given any
  other object $B$, morphism $g: B \to A_0$ and 2-isomorphism $\beta: f_1 g
  \implies f_0 g$, there is a unique morphism $h: B \to A$ such that
\[\xymatrix{
A \ar[r]^f & A_0  \\
B \ar[u]_h \ar[ru]_g
} \]
commutes, 
  \item in the above setting, the horizontal composition $\alpha: f_1g = f_1fh
\implies f_0fh = f_0g$ coincides with $\beta$,
  \item given a pair of morphisms $g_1, g_2: B \to A_0$, 2-isomorphisms
  $\beta_i: f_1 g_i \to f_0 g_i$, and corresponding morphisms $h_i: B \to
  A$ as above, there is a bijection between natural transformations $\gamma:
  h_1 \to h_2$ and natural transformations $\gamma': g_1 \to g_2$ such that
  we have a commutative diagram of natural transformations:
  \[\xymatrix{
    f_1g_1 \ar[r]^{f_0\gamma} \ar[d]_{\alpha g_1} & f_1g_2 \ar[d]^{\alpha g_2} 
\\
    f_0g_1 \ar[r]_{f_1\gamma} & f_0g_2
  }\]
  \end{enumerate}
\end{defn}
Note that the universal property is very strong, in that the diagrams are
required to commute on the nose, not up to 2-isomorphism. It follows that
the 2-equalizer, when it exists, is unique up to canonical isomorphism.

\begin{defn}[2-equalizers of functors]
  Let $\xymatrix{ C_0 \ar@<-.5ex>[r]_{\delta_1} \ar@<.5ex>[r]^{\delta_0} & C_1
  }$ be a pair of functors between (1-)categories. We define the
  \emph{equalizer\/} of $\delta_0$ and $\delta_1$, written $\eq(\delta_0,
  \delta_1)$, to be the category whose objects are pairs $(a, \phi)$ with $a
  \in \ob(C_0)$ and $\phi: \delta_1 a \to \delta_0 a$ is an isomorphism, and
  whose morphisms $(a, \phi) \to (a', \phi')$ are morphisms $f: a \to a'$ in
  $C_0$ such that the diagram
\[\xymatrix{
\delta_1 a \ar[r]^{\phi} \ar[d]_{\delta_1 f} & \delta_0 a \ar[d]^{\delta_0 f} \\
\delta_1 a' \ar[r]_{\phi'} & \delta_0 a'
}\]
commutes.
\end{defn}
It turns out that this is a 2-equalizer in the 2-category of
categories, and hence the notation is (relatively) unambiguous. 
\begin{lem} \label{2-cat limit}
  Let $\xymatrix{ C_0 \ar@<-.5ex>[r]_{\delta_1} \ar@<.5ex>[r]^{\delta_0} & C_1
  }$ be a pair of functors between (1-)categories. Then $\eq(\delta_0,
  \delta_1)$ is a 2-equalizer of $\delta_0, \delta_1$ in the 2-category
  of categories.
\end{lem}
\begin{proof}
  Let $C = \eq(\delta_0, \delta_1)$, and let $D$ be any category. We first need
  to show that functors $F: D \to C$ are in bijection with pairs $(F_0,
  \beta)$, where $F_0 : D \to C_0$ is a functor, and $\beta: \delta_1 F_0
  \to \delta_0 F_0$ is a natural isomorphism. 
  
  By defnition, a morphism $F: D \to C$ gives, for every object $d$ of $D$,
  a pair $Fd = (d_0, \phi_d)$, where $d_0 \in C_0$ and $\phi_d : \delta_1
  d_0 \to \delta_0 d_0$ is an isomorphism. Set $F_0 d = d_0$, and $\beta d = 
\phi_d :
  \delta_1 F_0 d \to \delta_0 F_0 d$. For a morphism $a: d \to d'$, $Fa:
  (F_0 d, \beta d) \to (F_0 d', \beta d')$ is a morphism in $C$, and by
  definition, this gives a morphism $F_0 a: F_0 d \to F_0 d'$ such that the
  diagram
  \[\xymatrix{
    \delta_1 F_0 d \ar[r]^{\delta_1 F_0 a} \ar[d]_{\beta d} & \delta_1 F_0
    d' \ar[d]^{\beta {d'}} \\
    \delta_0 F_0 d \ar[r]_{\delta_0 F_0 a} & \delta_0 F_0 d'
  }\]
  commutes. But this exactly says that $\beta$ is a natural isomorphism
  from $\delta_1 F_0$ to $\delta_0 F_0$.

  The above shows how to obtain a pair $(F_0, \beta)$ from a morphism $F: D
  \to C$. For the reverse, suppose we have a pair $F_0 : D \to C_0$ and a
  natural isomorphism $\beta: \delta_1 F_0 \to \delta_0 F_1$. Then for each
  $d$ in $D$, we have an isomorphism $\beta d : \delta_1 F_0 d \to \delta_0
  F_0 d$, which gives an object $(F_0 d, \beta d)$ of the 2-fiber product of
  functors $C$. For a morphism $a : d \to d'$, we have a commutative square
  \[\xymatrix{
    \delta_1 F_0 d \ar[r]^{\delta_1 F_0 a} \ar[d]_{\beta d} & \delta_1 F_0
    d' \ar[d]^{\beta d'} \\
    \delta_0 F_0 d \ar[r]_{\delta_0 F_0 a} & \delta_0 F_0 d',
  }\]
  which by definition gives a map to $C$.

  Let $G: C \to C_0$ be the canonical morphism, and $\alpha: \delta_1 G \to
  \delta_0 G$ the canonical natural transformation defined by $\alpha(c,
  \phi) = \phi$.

  Now consider a pair of functors $F, F': D \to C$, and corresponding
  pairs $(F_0, \beta), (F_0', \beta')$. We want to show that horizontal
  composition with $\alpha$ gives a
  bijection between natural transformations $F \to F'$ and natural
  transformations $\gamma: F_0 \to F_0'$ such that we have a commutative
  diagram
  \[\xymatrix{
    \delta_1 F_0 \ar[r]^{\delta_1 \gamma} \ar[d]_{\beta} & \delta_1 F_0'
    \ar[d]^{\beta'} \\
    \delta_0 F_0 \ar[r]_{\delta_0 \gamma} & \delta_0F_0'.
  }\]
  But this follows from the fact that by defintion of $\alpha$, $\alpha F = 
\beta$ and $\alpha F' = \beta'$.
\end{proof}

Further, we can
use this categorical 2-equalizer to express the universal property of
2-equalizers in general as follows:
\begin{lem} \label{2-equalizer represents}
Suppose that we have a $2$-category $\ms C$, and morphisms $f_0, f_1 \in
Hom_{\ms C}(A_0, A_1)$. Suppose $(A, f, \alpha)$ is a 2-equalizer of
$f_0, f_1$. Then for every object $B$ of $\ms C$, we have an isomorphism
(not just an equivalence!) of categories
\[ Hom_{\ms C}\left(B, A \right) \to \eq(Hom_{\ms C}(B,A_0), Hom_{\ms 
C}(B,A_1))\]
defined on objects by taking $h: B \to A$ to $(fh, \alpha h)$, where $\alpha h$ 
is the horizontal composition of $\alpha: f_1 f \to f_0 f$ with $h$.
\end{lem}
\begin{proof}
We note that the fact that the above extends to a fully faithful functor
follows from the previous lemma.

We will illustrate the inverse isomophism on the level of objects. Suppose we 
have an object in the equalizer
category on the right hand side. This consists of a morphism $g : B \to
A_0$ together with an isomorphism of morphisms (a 2-isomorphism) $\beta:
f_0g \implies f_1 g$. By definition, this gives a unique morphism $B
\to A$. 
\end{proof}
\begin{lem} \label{exact2eq} Suppose $\delta_0, \delta_1 : \ms A_0 \to \ms A_1$
  are a pair of functors between abelian categories. Let $\eta: \ms A \to \ms
  A_0$ be the 2-equalizer of these functors.
\begin{enumerate}
\item If $\delta_0, \delta_1$ are additive, then $\eta$ is additive and
  faithful.
\item If in addition $\delta_0, \delta_1$ are exact, then $\eta$ is faithfully
  exact (i.e. a sequence is short exact in $\ms A$ if and only if its image in
  $\ms A_0$ is short exact).
    \end{enumerate}
\end{lem}
\begin{proof}
  Suppose $\delta_0, \delta_1$ are additive. For objects $\underline a = (a,
  \phi), \underline a' = (a', \phi')$ of $\ms A$, by definition $Hom_{\ms
    A}(\underline a, \underline a')$ is the subgroup of $Hom_{\ms A_0}(a, a')$
  consisting of those $f: a \to a'$ such that $(\phi')(\delta_1 f) = (\delta_0
  f) (\phi)$. It follows that $\eta$ is additive and faithful.

  Suppose in addition that $\delta_0, \delta_1$ are exact. Given a morphism $f:
  \underline a' \to \underline a$, if we set $k, c$ to be the kernel and
  cokernel of $\eta f$ in $\ms A_0$, then considering the diagram:
\[\xymatrix{
0 \ar[r] & \delta_1 k \ar[r] & \delta_1 a' \ar[d]^{\phi'} \ar[r] & \delta_1 a 
\ar[d]^{\phi} \ar[r] & \delta_1 c \ar[r] & 0 \\
0 \ar[r] & \delta_0 k \ar[r] & \delta_0 a' \ar[r] & \delta_0 a \ar[r] & 
\delta_0 c \ar[r] & 0
}\]
We find that since the rows are exact (since $\delta_0, \delta_1$ are exact),
and the vertical maps are isomorphisms, there are unique isomorphisms $\psi_k:
\delta_1 k \to \delta_0 k$ and $\psi_c: \delta_1 c \to \delta_0 c$ which make
the diagram commute. One can then check that $(k, \psi_k)$ and $(c, \psi_c)$
(together with the above morphisms) are the kernel and cokernel in $\ms A$ of
the map $\underline a' \to \underline a$. It follows that formation of kernels
and cokernels commutes with $\eta$ (and hence images as well, viewed as the
cokernel of a kernel). It follows that $\eta$ is faithfully exact as claimed.
\end{proof}
\begin{rem} \label{exact remark} As a consequence of Lemma \ref{exact2eq}, if
  $\delta_0,
  \delta_1$ are both additive exact functors, and if $\ms B \to \ms A$ is an
  additive functor to the 2-equalizer, then it is (faithfully) exact if the
  associated functor $\ms B \to \ms A_0$ is (faithfully) exact.
\end{rem}

\subsection{2-equalizers of abelian monoidal categories}

\begin{lem}\label{monoidality}
  If $\delta_0,\delta_1:\ms C_0\to\ms C_1$ are monoidal functors
  between monoidal categories, then $\eq(\delta_0,\delta_1)$ has a
  canonical monoidal structure given by $(a, \phi) \otimes (a', \phi')
  = (a \otimes a', \phi \otimes \phi')$, such that the map $\delta:
  \eq(\delta_0, \delta_1) \to \ms C_0$ can be naturally extended to a
  monoidal functor, and such that the canonical natural transformation
  $\delta_1 \delta \to \delta_0 \delta$ is monoidal.

  Further, if $\ms C_0, \ms C_1$ have a braiding (resp. symmetric
  braiding), and $\delta_0, \delta_1$ are braided monoidal functors,
  then $\eq(\delta_0, \delta_1)$ has a natural braided
  (resp. symmetric braided) structure so that the morphism $\delta$ is
  braided monoidal.
\end{lem}
Said another way, the forgetful (2-)functor from monoidal (braided, resp. 
symmetric) categories to
categories, preserves 2-equalizers.
\newcommand{\xto}{\xrightarrow}
\newcommand{\sto}{\xto{\sim}}

In order to prove this lemma, we will begin by fixing some language
and notation. Following \cite{tensorcategories} when we say that $\ms
C_i$ is abelian monoidal ($i = 0, 1$) we mean that have a 6-tuple
$(\ms C_i, \otimes, \alpha, 1_i, \ell^i, r^i)$ consisting of additive
bifunctors $\otimes: \ms C_i \times \ms C_i \to \ms C_i$,
``associativity constraints,'' which are isomorphisms $\alpha_{a, b,
  c}^i: (a \otimes b) \otimes c \sto a \otimes (b \otimes c)$, natural
in $a, b, c$, unit objects $1_i \in \ms C_i$ with isomorphisms
$\ell_a^i: 1_i \otimes a \to a$, $r_a^i: a \otimes 1_i \to a$ both
natural in $a$, such that for every $a, b, c, d$ objects in $\ms C_i$,
we have a commutative pentagon
\begin{equation} \label{pentagon}
  \xymatrix @C=.1cm{ & & ((a \otimes b) \otimes c) \otimes d
    \ar[rrd]^{\alpha^i_{a \otimes b, c, d}} \ar[lld]_{\alpha^i_{a,b,c}
      \otimes \id_d} \\ (a \otimes (b \otimes c)) \otimes d
    \ar[rd]_{\alpha^i_{a, b\otimes c, d}} & & & & (a \otimes b)
    \otimes (c \otimes d) \ar[dl]^{\alpha^i_{a, b, c \otimes d}} \\ &
    a \otimes ((b \otimes c) \otimes d) \ar[rr]_{\id_a \otimes
      \alpha^i_{b, c, d}} & & a \otimes (b \otimes (c \otimes d)) }
\end{equation}
and such that for every $a, b$ objects in $\ms C$, we have a commutative 
triangle
\begin{equation} \label{unit compatibility}
  \xymatrix{
    (a \otimes 1_i) \otimes b  \ar[rr]^{\alpha^i_{a, 1_i, b}} \ar[rd]_{r^i_a 
\otimes \id_b} & & a \otimes (1_i \otimes b) \ar[ld]^{\id_a \otimes \ell_b^i} \\
    & a \otimes b
}
\end{equation}
In this case, we note \cite[Corollary~2.2.5]{tensorcategories} that we
have $\ell_{1_i}^i = r_{1_i}^i$, which we denote as $\iota_i : 1_i
\otimes 1_i \to 1_i$. To say that the functors $\delta_i: \ms C_0 \to
\ms C_1$ are monoidal is to say that we have specified pairs
$(\delta_i, J^i)$, where $\delta_i$ is an additive functor, and
$J^i_{a, b} : \delta_i(a) \otimes \delta_i(b) \xto \delta_i(a \otimes
b)$ are isomorphisms, natural in $a, b$ such that for each $a, b, c$
we have a commutative hexagon:
\begin{equation} \label{hexagon}
\xymatrix @C=.3cm @R=1cm{ (\delta_i(a) \otimes \delta_i(b)) \otimes
  d_i(c) \ar[rrrrr]^{\alpha_{\delta_i(a), \delta_i(b), \delta_i(c)}}
  \ar[d]_-{J^i_{a, b} \otimes \id_{\delta_i(c)}} & & & & & \delta_i(a)
  \otimes (\delta_i(b) \otimes \delta_i(c)) \ar[d]^-{\id_{\delta_i(a)
      \otimes J^i_{b, c}}} \\ \delta_i(a \otimes b) \otimes
  \delta_i(c) \ar[dr]_{J^i_{a \otimes b, c}} & & & & & \delta_i(a)
  \otimes \delta_i(b \otimes c) \ar[dl]^-{J^i_{a, b \otimes c}} \\ &
  \delta_i((a \otimes b) \otimes c) \ar[rrr]_{\delta_i(\alpha_{a, b,
      c})} & & & \delta_i(a \otimes (b \otimes c)), }
\end{equation}
and as in \cite[Remark~2.4.6]{tensorcategories} we require that there
is an isomorphism $\epsilon_i: 1_1 \to \delta_i(1_0)$ such that for
every object $a$ in $\ms C_0$, we have commutative squares
\doubleequation[left unit morphism, right unit morphism]{\xymatrix{
    1_1 \otimes \delta_i(x) \ar[r]^{\ell^1_{\delta_i(x)}}
    \ar[d]_{\epsilon_i \otimes \id_{\delta_i(x)}} & \delta_i(x)
    \\ \delta_i(1_0) \otimes \delta_i(x) \ar[r]^{J^i_{1_0, x}} &
    \delta_i(1_0 \otimes x) \ar[u]_{\delta_i(\ell_x^i)} }} {
  \xymatrix{ \delta_i(x) \otimes 1_1 \ar[r]^{r^1_{\delta_i(x)}}
    \ar[d]_{\id_{\delta_i(x)} \otimes \epsilon_i} & \delta_i(x)
    \\ \delta_i(x) \otimes \delta_i(1_0) \ar[r]^{J^i_{x, 1_0}} &
    \delta_i(x \otimes 1_0) \ar[u]_{\delta_i(r_x^i)} } } but which we
will simply consider as an identification $1_1 = \delta_i(1_0)$.

We say that $\omega$ is a monoidal natural transformation between
monoidal functors $(f, J), (f', J')$ if it is a natural transormation
from $f$ to $f'$ which commutes with $J, J'$ in the sense that we have
a commutative diagram for every $a, b$:
\[\xymatrix{
f(a) \otimes f(b) \ar[r]^{J_{a, b}} \ar[d]_{\phi(a) \otimes \phi(b)} & f(a 
\otimes b) \ar[d]^{\phi(a \otimes b)} \\
f'(a) \otimes f'(b) \ar[r]_{J'_{a, b}} & f'(a \otimes b)
}\]
\newcommand{\quo}[1]{\text{``}{#1}\text{''}}
\begin{proof}[Proof of Lemma~\ref{monoidality}]
  We define a monoidal structure on $\eq(\delta_0, \delta_1)$ by
  setting $(a, \phi) \otimes (b, \psi) = (a \otimes b, \quo{\phi
    \otimes \psi})$, where $\quo{\phi \otimes \psi}$ is defined via
  the commutative square
  \begin{equation} \label{eqtensor}
      \xymatrix{
      \delta_1(a \otimes b) \ar[d]_{\quo{\phi \otimes \psi}} & \delta_1(a) 
\otimes \delta_1(b) \ar[l]_{J^1_{a, b}} \ar[d]^{\phi \otimes \psi} \\
      \delta_0(a \otimes b)  & \delta_0(a) \otimes \delta_0(b).
      \ar[l]^{J^0_{a \otimes b}}}
  \end{equation}
  We claim that the morphisms
  \[\alpha_{(a, \phi), (b, \psi), (c, \theta)} : (a, \phi) \otimes ((b, \psi) 
\otimes (c, \theta)) \to ((a, \phi) \otimes (b, \psi)) \otimes (c, \theta) \]
  given by the associativity constraint for $\ms C_0$, which we write
  as $\alpha_{a, b, c}: a \otimes (b \otimes c) \to (a \otimes b)
  \otimes c$, constitutes an associativity constraint for
  $\eq(\delta_0, \delta_1)$. Note that this makes sense as a
  definition, because $\delta$ is a faithful functor -- we need only
  check that this is a valid morphism in the equalizer category. By
  definition, we may write
  \begin{align*}
  (a, \phi) \otimes ((b, \psi) \otimes (c, \theta)) &= (a \otimes (b \otimes 
c), \quo{\phi \otimes (\quo{\psi \otimes \theta})}) \\
  ((a, \phi) \otimes (b, \psi)) \otimes (c, \theta) &= ((a \otimes b) \otimes 
c, \quo{(\quo{\phi \otimes \psi}) \otimes \theta}) \\
  \end{align*}
  where $\quo{\phi \otimes (\quo{\psi \otimes \theta})}$ and
  $\quo{(\quo{\phi \otimes \psi}) \otimes \theta}$ are defined by the
  diagrams
  \[
  \xymatrix @C = 1.4cm{ \delta_1(a \otimes (b \otimes c))
    \ar[d]_{\quo{\phi \otimes (\quo{\psi \otimes \theta})}} &
    \delta_1(a) \otimes \delta_1((b \otimes c)) \ar[l]_{J^1_{a, b
        \otimes c}} \ar[d]^{\phi \otimes (\quo{\psi \otimes \theta})}
    & & \delta_1(a) \otimes (\delta_1(b) \otimes \delta_1(c))
    \ar[d]^{\phi \otimes (\psi \otimes \theta)}
    \ar[ll]_{\id_{\delta_1(a)} \otimes J^1_{b, c}} \\ \delta_0(a
    \otimes (b \otimes c)) & \delta_0(a) \otimes \delta_0(b \otimes c)
    \ar[l]^{J^0_{a, b \otimes c}} & & \delta_0(a) \otimes (\delta_0(b)
    \otimes \delta_0(c)) \ar[ll]^{\id_{\delta_0(a)} \otimes J^0_{b,
        c}} }
  \]
  \[
  \xymatrix @C = 1.4cm{
    \delta_1((a \otimes b) \otimes c)
      \ar[d]_{\quo{(\quo{\phi \otimes \psi}) \otimes \theta}}
      &
    \delta_1(a \otimes b) \otimes \delta_1(c)
      \ar[l]_{\delta_0(a \otimes b) \otimes \delta_0(c)}
      \ar[d]^{(\quo{\phi \otimes \psi}) \otimes \theta}
      & &
    (\delta_1(a) \otimes \delta_1(b)) \otimes \delta_1(c)
      \ar[d]^{(\phi \otimes \psi) \otimes \theta}
      \ar[ll]_{ J^1_{a, b} \otimes \id_{\delta_1(c)}}
    \\
    \delta_0((a \otimes b) \otimes c)
      &
    \delta_0(a \otimes b) \otimes \delta_0(c)
      \ar[l]^{J^0_{a \otimes b, c}}
      & &
    (\delta_0(a) \otimes \delta_0(b)) \otimes \delta_0(c)
      \ar[ll]^{J^0_{a, b} \otimes \id_{\delta_0(c)}}
  }
  \]
and we need to check that the proposed associativity constraint
induces a morphism in the equalizer category, which is to say that the
following diagram commutes:
\[\xymatrix @C=2cm{
\delta_1(a \otimes (b \otimes c)) \ar[r]^{\delta_1(\alpha_{a, b, c})}
\ar[d]_{\quo{\phi \otimes (\quo{\psi \otimes \theta})}} & \delta_1((a
\otimes b) \otimes c) \ar[d]^{\quo{(\quo{\phi \otimes \psi}) \otimes
    \theta}} \\ \delta_0(a \otimes (b \otimes c))
\ar[r]_{\delta_0(\alpha_{a, b, c})} & \delta_0((a \otimes b) \otimes
c) }\] But expanding out the vertical arrows in this diagram, this is
equivalent to observing that the following diagram commutes:
\[\xymatrix @C = 3cm{
  \delta_1(a \otimes (b \otimes c))
    \ar[r]^{\delta_1(\alpha_{a, b, c})}
    \ar@/_10pc/[ddddd]^{\quo{\phi \otimes (\quo{\psi \otimes \theta})}}
    &
  \delta_1((a \otimes b) \otimes c)
    \ar@/^10pc/[ddddd]_{\quo{(\quo{\phi \otimes \psi}) \otimes \theta}}
  \\
  \delta_1(a) \otimes (\delta_1(b \otimes c))
    \ar[u]^{J^1_{a, b \otimes c}}
    &
  \delta_1(a \otimes b) \otimes \delta_1(c)
    \ar[u]_{J^1_{a \otimes b, c}}
  \\
  \delta_1(a) \otimes (\delta_1(b) \otimes \delta_1(c))
    \ar[u]^{\id_{\delta_1(a)} \otimes J^1_{b, c}}
    \ar[d]_{\phi \otimes (\psi \otimes \theta)}
    \ar[r]^{\alpha_{\delta_1(a), \delta_1(b), \delta_1(c)}}
    &
  (\delta_1(a) \otimes \delta_1(b)) \otimes \delta_1(c)
    \ar[u]_{J^1_{a, b} \otimes \id_{\delta_1(c)}}
    \ar[d]^{(\phi \otimes \psi) \otimes \theta}
  \\
  \delta_0(a) \otimes (\delta_0(b) \otimes \delta_0(c))
    \ar[d]_{\id_{\delta_0(a)} \otimes J^0_{b, c}}
    \ar[r]_{\alpha_{\delta_0(a), \delta_0(b), \delta_0(c)}}
    &
  (\delta_0(a) \otimes \delta_0(b)) \otimes \delta_0(c)
    \ar[d]^{J^0_{a, b} \otimes \id_{\delta_0(c)}}
  \\
  \delta_0(a) \otimes \delta_0(b \otimes c)
    \ar[d]_{J^0_{a, b \otimes c}}
    &
  \delta_0(a \otimes b) \otimes \delta_0(c)
    \ar[d]^{J^0_{a \otimes b, c}}
  \\
  \delta_0(a \otimes (b \otimes c))
    \ar[r]_{\delta_0(\alpha_{a, b, c})}
    & \delta_0((a \otimes b) \otimes c)}\] But this diagram commutes since the
  top and bottom portions are the compatibility hexagons of
  Diagram~\ref{hexagon} for the monoidal functors $(\delta_1, J^1)$ and
  $(\delta_0, J^0)$ respectively, and the middle square commutes due to the
  naturality of $\alpha$ in its three variables.

It follows that $\eq(\delta_0, \delta_1)$ has a monoidal structure (note that by
convention with units, we are implicitly identifying $1 \in \eq(\delta_0,
\delta_1)$ with $(1_0, \epsilon_0^{-1} \epsilon_1^{-1})$). The fact that the
associativity constraint satisfies the pentagon condition of
diagram~\ref{pentagon} follows from the fact that $\delta$ is faithful, and it
similarly follows that the left and right unit morphisms satisfy the condition
of diagram~\ref{unit compatibility}.

We may extend $\delta: \eq(\delta_0, \delta_1) \to \ms C_0$ to a monoidal
functor $(\delta, J)$, by defining $J_{a, b}: \delta(a, \phi) \otimes (b, \psi)
= a \otimes b \to \delta((a, \phi) \otimes (b, \psi)) = a \otimes b$ to be the
identity morphism (and $\epsilon: 1_0 \to \delta(1_0, \epsilon_0^{-1}
\epsilon_1^{-1}) = 1_0$ to also be the identity morphism). The fact that the
compatibility hexagon ensuring that this defines a morphism holds (similarly to
diagram~\ref{hexagon}) is immediate. Finally, the fact that the canonical
natural transformation $\delta_1 \delta \to \delta_2 \delta$ is monoidal follows
from the definition of the monoidal structure as in diagram~\ref{eqtensor}.

If in addition, the categories $\ms C_i$ are given a braiding defined by a
natural isomoprhism $C^i_{a, b} : a \otimes b \to b \otimes a$, then one can
check that $C^0$ induces isomorphisms
\[C_{(a, \phi), (b, \psi)} : (a, \phi) \otimes (b, \psi) \to (b, \psi) \otimes
(a, \phi) \] provided that $\delta_0, \delta_1$ are braided morphisms (see
\cite[Section~8.1]{tensorcategories} for precise definitions). Consequently, it
again follows quickly from the fact that $\delta$ is faithful that $\delta$ is a
morphism of braided monoidal categories. Again, by faithfulness, it is easy to
see that if the braiding $C^0$ on $\ms C_0$ is symmetric, i.e. $C^0_{b, a}
C^0_{a, b} = \id_{a \otimes b}$, then so is the braiding $C$ on $\eq(\delta_0,
\delta_1)$.
\end{proof}

\newcommand{\ctcat}[1]{{#1}_\otimes^{\text{r-Ex}}}
\newcommand{\cohtensorhom}[2]{\rHom(\Coh\left({#1}), \Coh({#2})\right)}
\newcommand{\cthom}[2]{\left[\Coh(#1), \Coh({#2})\right]_{\otimes}}

Suppose we are given symmetric monoidal categories $\cA, \cB$.
Let $\rHom(\cA,\cB)$ denote the category whose objects are right exact 
symmetric monoidal functors and whose morphisms are natural isomorphisms.

\begin{lem}
Let $\cA_0, \cA_1$ be abelian, symmetric monoidal categories. Suppose we are 
given functors $d^0, d^1 \in \rHom(\cA_0, \cA_1)$. Let $\cA$ be the 2-equalizer 
of these functors. Then the natural functor $\cA \to \cA_0$ is in $\rHom(\cA, 
\cA_0)$.
\end{lem}
\begin{proof}
This is an immediate consequence of Lemma~\ref{exact2eq}.
\end{proof}
\newcommand{\reasmc}{\mathrm{rExAbSMon}}
\newcommand{\ab}{\mathrm{Ab}}
Let $\reasmc$ denote the $2$-category of (small) abelian symmetric monoidal 
categories with morphisms being right exact symmetric monoidal functors and 
$2$-morphisms being natural isomorphisms of functors. Let $\ab$ the 
$2$-category of (small) abelian categories. As a corollary to the above, we 
have the following.
\begin{cor}
Suppose that $d^0, d^1 \in \rHom(\cA, \cB)$. Then the 2-equalizer in $\ab$ of 
these functors is naturally a symmetric monoidal category, and coincides with 
the 2-equalizer in $\reasmc$.
\end{cor}

\subsection{Patching with respect to a fibered category}

\begin{defn}
Let $F \to C$ be a fibered category, and suppose that we have a diagram of 
morphisms $U_\bullet$ in $C$:
\[U_\bullet = \left[\xymatrix{
U_1 \ar@<-.5ex>[r]_{d_1} \ar@<.5ex>[r]^{d_0} & U_0
}\right]\]
Choosing pullbacks $d_i^*$ gives us functors (well defined up to natural 
equivalence)
\[
\xymatrix{
F(U_0) \ar@<-.5ex>[r]_{d_1^*} \ar@<.5ex>[r]^{d_0^*} & F(U_1)
}
\]
and we define $F(U_\bullet)$ to be the 2-equalizer of this diagram.
\end{defn}
\begin{notn}
If $U_\bullet$ is a diagram as above, we write $d: U_\bullet \to U$ for an
object $U$ to mean a morphism $d: U_0 \to U$, such that $dd_1 = dd_0$.
We will refer to such a diagram as a \emph{patching context}.
\end{notn}

\begin{defn}
Given $d: U_\bullet \to U$, we have a functor (well defined up to natural 
isomorphism)
\[F(U) \to F(U_\bullet)\]
given by taking an object $a$ of $F(U)$ to $(d^* a, d_1^*d^*a 
\overset{\iota}\to d_0^*d^*a)$, where $\iota$ is the unique morphism of $F$ 
such that we have a commutative diagram:
\[\xymatrix{
d_1^* d^* a \ar[rr] \ar[dr]_{\cong} \ar[dd]_{\exists{!}\cong} & & d^*a \ar[dr] 
\\
& (d d_1)^* a \ar[rr] & & a \\
d_0^* d^* a \ar[rr] \ar[ur]^{\cong} & & d^* a, \ar[ru]
}\]
we say that \emph{patching holds for $F$ with respect to the patching
context $U_\bullet \to U$\/} if $F(U) \to F(U_\bullet)$ is an equivalence of 
categories.
\end{defn}

In particular, we can regard a presheaf over $C$ as a fibered category where 
each $F(a)$ is a set for $a \in \ob(C)$ (i.e. no nontrivial morhisms). In this 
case, we note that patching holding for $U_\bullet \to U$ is the statement that
$F(U) \to F(U_0)$ is the equalizer of $\delta_0^*, \delta_1^*: F(U_0) \to 
F(U_1)$.

\section{Beyond coherent sheaves}

In this section, we will show that if patching holds for coherent sheaves on a 
system of locally excellent algebraic stacks, then patching also holds for 
morphisms to any sufficiently nice algebraic stack.
\begin{thm} \label{tannaka main}
Let $\cX, \cX_1, \cX_2$ be locally excellent algebraic stacks, and suppose that 
we have morphisms $\cX_\bullet = \left[
\xymatrix{
\cX_1 \ar@<-.5ex>[r]_{d_1} \ar@<.5ex>[r]^{d_0} & \cX_2
}\right]$ and $\cX_\bullet \to \cX$. Suppose further that patching holds for 
coherent sheaves with respect to $\cX_\bullet \to \cX$. Let $\cG$ be a 
Noetherian algebraic stack with affine stabilizers. Then patching also holds 
for $Hom(\_, \cG)$ with respect to $\cX_\bullet \to \cX$.
\end{thm}

\begin{cor}
Suppose we are given $\cX_\bullet \to \cX$ as in Theorem~\ref{tannaka main} for 
which patching holds for coherent sheaves. Then for $G$ an affine group scheme 
over $\cX$, we have a $6$-term exact sequence:
\[\xymatrix{
0 \ar[r] & H^0(\cX,G) \ar[r] & H^0(\cX_0,G) \ar[r] & H^0(\cX_1
,G) \ar@<-2pt> `d[l]
`[lld] [lld] \\
& H^1(\cX, G) \ar[r] & H^1(\cX_0,G) \ar@<.5ex>[r]^{d_0^*}
\ar@<-.5ex>[r]_{d_1^*} & H^1(\cX_1,G)
}\]
where the map $H^0(\cX_0, G) \to H^0(\cX_1, G)$ is given by $g \mapsto d_1^*(g) 
d_0^*(g^{-1})$. In the case that $G$ is abelian, this can be interpreted as an 
exact sequence of abelian groups:
\[\xymatrix{
0 \ar[r] & H^0(\cX,A) \ar[r] & H^0(\cX_0,A) \ar[r] & H^0(\cX_1
,A) \ar@<-2pt> `d[l]
`[lld] [lld] \\
& H^1(\cX, A) \ar[r] & H^1(\cX_0,A) \ar[r] & H^1(\cX_1,A)
}\]
\end{cor}
\begin{proof}
Exactness of the top row is exactly the statement of Theorem~\ref{tannaka main} 
in the case $\cG$ is the scheme $G$. We define the connecting map $H^0(\cX_1, 
G)$ as follows.
For $g \in H^0(\cX_1, G) = G(\cX_1)$, consider $G$-torsor on $\cX_\bullet$ 
(i.e. the object of $BG(\cX_\bullet)$) described by $(1, g)$, where $1$ denotes 
the trivial $G$-torsor on $\cX_0$, and $g$ is considered as an automorphism of 
the trivial $G$-torsor on $\cX_1$.
By Theorem~\ref{tannaka main}, this gives a $G$-torsor on $\cX$, well 
defined up to isomorphism, and hence a class of $H^1(\cX, G)$, which by 
construction is trivial when restricted to $\cX_0$. This is the definition of 
the connecting map, and why it maps into the pointed kernel. To see that it is 
the entire pointed kernel, if $P$ is a $G$-torsor on $\cX$ which is trivial 
over $\cX_0$, then Theorem \ref{tannaka main} implies that $BG(\cX) \cong BG(\cX_\bullet)$, and we have that the image of $P$ 
in $BG(\cX_\bullet)$ must have the form $(1, g)$ as above.

Finally, it is fairly direct to see that the image of $H^1(\cX, G) \to 
H^1(\cX_0, G)$ lies in the equalizer. Conversely, suppose that $[P] \in 
H^1(\cX_0, G)$ is in the equalizer of the map to $H^1(\cX_1, G)$. In this case, 
we may find some isomorphism $\phi: d_1^*P \to d_0^*P$. But in this case, it is 
clear that the $G$-torsor $\til P$ which must exist and map to the object $(P, 
\phi)$, showing exactness of the bottom row.
\end{proof}

Before giving the proof of Theorem~\ref{tannaka main}, let us set up a bit of 
notation.
Let $\cG, \cX$ be locally Noetherian algebraic stacks and
let $\Coh(\cG)$ and $\Coh(\cX)$ denote their categories of coherent sheaves.
Note that these are symmetric monoidal categories with respect to the tensor 
product.
A morphism $f \colon \cX \rightarrow \cG$ induces a functor
\begin{align*}
  f^* \colon \Coh(\cG) \rightarrow \Coh(\cX)
\end{align*}

Taking a map of algebraic stacks $f: \cX \to \cG$ to $f^*$, we obtain a functor
\begin{align*}
   \Hom\left( \cX,\cG \right) \rightarrow \rHom (\Coh(\cG),\Coh(\cX)).
\end{align*}
via pullback.

For a fixed algebraic stack $\cG$, we define a fibered category $\ctcat{\cG}$ 
over the category of algebraic stacks as follows. The objects of $\ctcat{\cG}$ 
are pairs $(\cX, F)$ where $\cX$ is an algebraic stack and $F$ is an object of 
$\cohtensorhom{\cG}{\cX}$. A morphism $(\cX', F') \to (\cX, F)$ consists of a 
morphism $f: \cX' \to \cX$ together with a morphism $F' \to f^* F$.
The association $f \mapsto f^*$ gives a morphism of fibered categories $\cG \to 
\ctcat{\cG}$, where we write $\cG$ for the representable fibered category it 
defines.

In \cite{HallRydh}, Hall and Rydh prove that this morphism is an equivalence 
under certain
conditions:

\begin{thm}[Theorem 1.1 in \cite{HallRydh}]
  \label{tannaka}
  Let $\cG$ be a Noetherian algebraic stack with affine stabilizers.
  Then, for every locally excellent algebraic stack $\cX$, the natural functor
  \begin{align*}
    \cG(\cX) = \Hom\left( \cX,\cG \right) \rightarrow \rHom 
(\Coh(\cG),\Coh(\cX)) = \ctcat{\cG}(\cX)
  \end{align*}
  is an equivalence.
\end{thm}
\begin{proof}[Proof of Theorem~\ref{tannaka main}]
Since patching holds for coherent sheaves for $\cX_\bullet \to \cX$, we have an 
equivalence of categories:
\[\ctcat{\cG}(\cX) = \rHom(\Coh(\cG), \Coh(\cX)) \cong \rHom(\Coh(\cG), 
\Coh(\cX_\bullet)))\]
and by Lemma~\ref{2-equalizer represents} we can identify this last category 
with the 2-equalizer
\[
\eq\left(
\xymatrix{
\rHom(\Coh(\cG), \Coh(\cX_0)
\ar@<-.5ex>[r] \ar@<.5ex>[r] & \rHom(\Coh(\cG), \Coh(\cX_1))
}
\right)\]
But this in turn is by definition the category $\ctcat{\cG}(\cX_\bullet)$, 
showing that patching holds.
\end{proof}
\subsection{Examples of patching for coherent sheaves}

We will now record, some contexts in which patching results are known to hold 
for categories of coherent sheaves, which thereby gives patching for maps to 
Noetherian algebraic stacks with affine stabilizers. In each of the following 
examples, we will let $\Coh$ denote the stack of coherent sheaves.

\subsubsection{Formal patching}

\begin{thm}\!\!$($\cite[FR, Prop.~4.2]{FR}$)$\label{formal coherent}
Let $X$ be a Noetherian, 1-dimensional scheme, $Z \subset X$ a finite subset of 
closed points, $U \subset X$ its open complement.
For a closed point $\xi \in Z$, let $X_\xi = \Spec(\wh{\ms O}_{X, \xi})$, 
$K_\xi$ the fraction field of $\wh{\ms O}_{X, \xi}$, $X' = \coprod_{\xi \in Z} 
X_\xi$ and $U' = \coprod_{\xi \in Z} \Spec(K_\xi)$. Then patching holds 
for coherent sheaves with respect to the patching context
\[
\xymatrix{
U' \ar@<-.5ex>[r]_-{d_1} \ar@<.5ex>[r]^-{d_0} & X' \times U \ar[r]^-{d} & X.
}
\]
\end{thm}

\begin{cor} \label{formal patch}
In the situation of Theorem~\ref{formal coherent}, patching holds for morphisms 
to Noetherian stacks with affine stabilizers. In particular, patching holds 
for categories of $G$-torsors for any linear algebraic group $G$.
\end{cor}

\subsubsection{Thickened formal patching}

\begin{notn} \label{ring patching setup}
Let $T$ be a complete discrete valuation ring with uniformizer $t$, and
let $\mc X$ be a proper $T$-curve with reduced closed fiber $X$. Let $F$ be the 
function field of $\mc X$.
Suppose that $\mc P \subset X$ a finite subset of closed points, such that $X 
\setminus Z$ is a disjoint union of connected affine components.
Let $\mc U$ be the set of irreducible components of $X \setminus Z$.

For any connected affine open $W \subset X$, let $R_W$ denote the subring of 
elements of $F$ which are regular at every point of $W$ -- that is to say
\[R_U = \cap_{x \in W} \ms O_{\mc X, x}.\]
Let $\wh R_W$ be the $t$-adic completion $R_W$. By \cite{HHK:Weier}, $\wh R_W$ 
is a domain, and we let $F_W$ denote its fraction field.

For $P \in X$, a closed point, let $\wh R_P$ be the complete local ring of $\mc 
X$ at $P$. Finally, if $\wp$ is a height one prime of $\wh R_P$ lying over 
$t\wh R_P$, let $R_\wp$ be the localization of $\wh R_P$ at $\wp$ and $\wh R_P$ 
its $\wp$-adic completion (also coinciding with its $t$-adic completion). We 
let $\mc B$ denote the set of all such height one primes $\wp$.
Note that these are the branches of the closed subschemes $\overline U \subset 
\mc X$ at the points $P \in \mc P$.

We note (see for example \cite[Section~3.1]{HHK}), that whenever $\wp$ is a 
branch along $P$, there are natural inclusions of rings $\wh R_P \to \wh 
R_\wp$, and when $P$ is in the closure of a component $W$ of the reduced closed 
fiber, for each branch $\wp$ along $W$ (i.e., cut out by the ideal defining the 
closed set $\overline W$ in $\mc X$), there is an inclusion $\wh R_W \to \wh 
R_\wp$. Taken together, these give natural maps
\[\xymatrix{
\prod_{P \in \mc P} \wh R_P \ar[r] & \prod_{\wp \in \mc B} \wh R_P & \ar[l] 
\prod_{U \in \mc U} \wh R_U
}\]
\end{notn}

\begin{thm}\!\!$(${\rm see }\cite[Pries, Theorem~3.4]{Pries}$)$ \label{thickened coherent}
In the language of Notation~\ref{ring patching setup}, patching holds for 
coherent sheaves with respect to the patching context
\[\xymatrix{
\Spec\left(\prod_{\wh \in \mc B} \wh R_\wp\right) 
\ar@<.5ex>[r]^-{d_0} \ar@<-.5ex>[r]_-{d_1} 
& 
\Spec\left(\prod_{P \in \mc P} \wh R_P\right) \coprod \Spec\left(\prod_{U
\in \mc U} \wh R_U\right) 
\ar[r] 
&
\mc X
}\]
\end{thm}
\begin{proof}
This is precisely the statement of \cite{Pries} with the additional assumption 
that the schemes $U$ are affine, which allows us to replace the categories of 
modules over the formal completions with the modules over the corresponding 
complete rings.
\end{proof}

\begin{cor} \label{thickened patch}
In the situation of Theorem~\ref{thickened coherent}, patching holds for 
morphisms to Noetherian stacks with affine stabilizers. In particular, 
patching holds for categories of $G$-torsors for any linear algebraic group $G$.
\end{cor}

\subsubsection{Field patching}
\begin{thm}\!\!$($\cite[Prop.~3.9]{HHK:pop}$)$ \label{field patching}
In the language of Notation~\ref{ring patching setup}, suppose we are given an 
open affine connected subset $W \subset X$ and a finite collection of closed 
points $\mc Q \subset W$. Let $\mc V$ be the set of connected components of $W 
\setminus \mc Q$, and let $\mc B$ denote the collection of branches along $W$ 
at the points $Q \in \mc Q$. Then patching holds for coherent sheaves with 
respect to the patching context
\[\xymatrix{
\Spec\left(\prod_{\wh \in \mc B} F_\wp\right) 
\ar@<.5ex>[r]^-{\pi_1} \ar@<-.5ex>[r]_-{\pi_2} 
&
\Spec\left(\prod_{Q \in \mc Q} F_Q\right) \coprod \Spec\left(\prod_{V \in
\mc V} F_V\right) 
\ar[r] 
& 
\Spec F_W
}\]
\end{thm}

\begin{thm}\!\!$($\cite[Prop.~3.10]{HHK:pop}$)$ \label{field patching small}
In the language of Notation~\ref{ring patching setup}, suppose we are given
a proper birational morphism $f: Y \to X$ and a closed point
$P \in \mc P$. Let $V \subset Y$ be inverse image of $P$ in $Y$, and let
$\til X$ be the proper transform of $X$. Suppose that $\dim(V) = 1$, and
that $f$ restricts to an isomorphism $\ms Y \setminus V \to \ms X \setminus
\{P\}$. Choose $\mc Q$ a finite collection of closed points $\mc P$ of $V$
including all the points of $V \cap \til X$.  Let $\mc W$ be the set of
connected components of $V \setminus \mc P$, and let $\mc B'$ be the set of
branches at the points in $\mc P$ along the components of $V$. 
Then patching holds for
coherent sheaves with respect to the patching context
\[\xymatrix{
\Spec\left(\prod_{\wh \in \mc B'} F_\wp\right) 
\ar@<.5ex>[r]^-{\pi_1} \ar@<-.5ex>[r]_-{\pi_2} 
&
\Spec\left(\prod_{Q \in \mc Q} F_Q\right) \coprod \Spec\left(\prod_{U \in
\mc W} F_U\right) 
\ar[r] 
& 
\Spec F_P
}\]
\end{thm}

\begin{cor} \label{pop cor}
In the situation of 
Theorems~\ref{field patching} and~\ref{field patching small}, 
patching holds for morphisms to Noetherian stacks with affine
stabilizers. In particular, patching holds for categories of $G$-torsors
for any linear algebraic group $G$ defined over $F_W$ and $F_P$ in the
respective contexts.
\end{cor}

We note that this was known to hold for groups which were defined over $F$,
the function field of $\mc X$, but not necessarily for groups over $F_W$
and $F_P$.

\section{Relative categories of coherent sheaves}

The main result of the section is the following.

\begin{thm} \label{relative patch}
Suppose we are given morphisms $$X_1\arr{d_0}{d_1}X_0\to X$$ of  
excellent Noetherian schemes such that patching holds for coherent sheaves with 
respect to $X_\bullet \to X$.
Let $S \to X$ be a proper algebraic space, and 
let $S_\bullet=S\times_{X} X_\bullet$ be the associated diagram of spaces induced by pullback. Then patching holds for coherent sheaves with respect to $S_\bullet 
\to S$ (and hence it also holds for morphisms to Noetherian algebraic stacks 
with affine stabilizers by Theorem \ref{tannaka main}).
\end{thm}

The proof of this theorem will occupy the remainder of the section.

\begin{lem} \label{fflat}
The map $S_0 \to S$ is faithfully flat.
\end{lem}
\begin{proof}
By Lemma~\ref{exact2eq} and the fact that pullback induces an isomorphism $\Coh(X) \simto \Coh(X_\bullet)$, 
it follows that the pullback map $\Coh(X) \to \Coh(X_0)$ is faithfully 
exact, which tells us that $X_0 \to X$ is faithfully flat. But since $S_0 
\to S$ is obtained from this by pullback, it is also faithfully flat.
\end{proof}

\begin{lem} \label{fexact}
The functor $\Coh(S) \to \Coh(S_\bullet)$ commutes with the formation of 
kernels and cokernels (i.e. is faithfully exact).
\end{lem}
\begin{proof}
Since $S_0 \to S$ is faithfully flat, we have that $\Coh(S) \to \Coh(S_0)$ is 
faithfully exact. The conclusion follows from Remark~\ref{exact remark}.
\end{proof}

\begin{lem} \label{fullyfaithful}
The functor $\Coh(S) \to \Coh(S_\bullet)$ is fully faithful.
\end{lem}
\begin{proof}
Suppose we have coherent sheaves $W, V$ on $S$. Consider the (additive) group 
scheme $G$ on $X$ whose values on $T \to X$ are $\Hom_{\ms O_{S_T}}(W_T, 
V_T)$.
Thinking of this as a stack (with trivial inertia), Theorem~\ref{tannaka main} 
gives a 2-equalizer diagram (of setoids with abelian group structure)
\[\arrr{G(X)}{G(X_0)}{G(X_1)}{}{}\]
which can be identified with an isomorphism of abelian groups
\[\Hom_{S}(W, V) \simto \Hom_{S_\bullet}(W_{S_\bullet}, V_{S_\bullet})\]
\end{proof}

\begin{lem} \label{eimage-lemma}
The essential image of $\Coh(S) \to \Coh(S_\bullet)$
\begin{enumerate}
    \item \label{tclosed} is closed under $\otimes$,
    \item \label{pclosed} contains the image of $\Coh(X_\bullet)$ under 
pullback,
    \item \label{cclosed} is closed under formation of kernels and cokernels, and
    \item \label{eclosed} is closed under extensions.
\end{enumerate}
\end{lem}
\begin{proof}
Part \ref{tclosed} follows from Remark~\ref{monoidality}. Part \ref{pclosed} 
follows from the fact that we have a commutative diagram
\[\xymatrix{
\Coh(X) \ar[r] \ar[d] & \Coh(X_0) \ar[d] \ar@<-.5ex>[r] \ar@<.5ex>[r] & 
\Coh(X_1) \ar[d] \\
\Coh(S) \ar[r] & \Coh(S_0) \ar@<-.5ex>[r] \ar@<.5ex>[r] & \Coh(S_1). \\
}\]
For part \ref{cclosed}, we will consider the case of cokernels (the case of 
kernels is similar). Suppose we are given a right exact sequence
\[ F \to G \to C \to 0\]
in $\Coh(S_\bullet)$, with $F, G$ in the image of $F', G'$ in $\Coh(S)$. Let
\[F' \to G' \to C' \to 0\] be right exact in $\Coh(S)$ (i.e. $C'$ is the
cokernel in $\Coh(S)$). Since $\Coh(S) \to \Coh(S_\bullet)$ is faithfully exact
(by Lemma~\ref{fexact}), it follows that we have an isomorphism between $C$ and
the image of $C'$ showing that $C$ is in the essential image as desired.

Using \cite[Thm~3.5]{Yoneda} or 
\cite{Oort:Yoneda}, part~\ref{eclosed} follows so long as formation of $\Ext$ groups 
commutes with the functor $\Coh(S) \to \Coh(S_\bullet)$.
But this follows from the fact that the functor is fully faithful (and hence 
preserves the $\Hom$ functor) and is faithfully exact (and hence commutes with the 
construction of the derived functor).
\end{proof}

\begin{lem}\label{lem:projective case described above}
If $S$ is a projective $X$-scheme then $\Coh(S) \to 
\Coh(S_\bullet)$ is essentially surjective.
\end{lem}
\begin{proof}
Let $V_\bullet = (V, \phi)$ be an object of $\Coh(S_\bullet)$. Let $\pi: S 
\to X$ be the structure morphism. We will write $\pi_*$ and $\pi^*$ for the 
standard pushforward and pullback maps, as well as for the natural functors 
they induce between the categories $\Coh(S_\bullet)$ and $\Coh(X_\bullet)$.

By tensoring with image of the relatively ample invertible sheaf $\mc O_{S}(n)$, we can define 
$V_\bullet(n)$ for any $n \in \mathbf Z$. For any such integer $n$, note that 
we have a canonical morphism $\pi^* \pi_* V_\bullet(n) \to V_\bullet(n)$, and 
since the functor $\Coh(S_\bullet) \to \Coh(S_0)$ is faithfully exact by 
Lemma~\ref{fexact}, it follows that this is surjective for $n\gg 0$.
Let $K_\bullet$ be the kernel of this morphism. Again, for some $m \gg 0$, we 
have a surjection $\pi^* \pi_* K_\bullet(m) \to K_\bullet(m)$. Taken together, 
this gives a right exact sequence:
\[ \pi^* \pi_* K_\bullet(m) \to (\pi^* \pi_* V_\bullet(n))(m) \to V_\bullet(n + 
m). \]
By Lemma~\ref{eimage-lemma}(\ref{tclosed},\ref{pclosed}), $\pi^* \pi_* 
K_\bullet(m)$ and $(\pi^* \pi_* V_\bullet(n))(m)$ are in the essential image.
By Lemma~\ref{fullyfaithful}, the morphism between them is as well, and by 
Lemma~\ref{eimage-lemma}(\ref{cclosed}), so is their cokernel. Therefore 
$V_\bullet(n + m)$ is in the essential image.
But since we may tensor with $\mc O(-n - m)$ by 
Lemma~\ref{eimage-lemma}(\ref{tclosed}) and stay in the essential image as 
well, this says that $V_\bullet$ is in the essential image as claimed.
\end{proof}

\begin{proof}[Proof of Theorem~\ref{relative patch}]
We proceed by Noetherian induction (on $\cS$), the case of $\cS=\emptyset$ being clear. By Chow's Lemma \cite[\href{https://stacks.math.columbia.edu/tag/088U}{Tag 088U}]{stacks-project}, there is a proper birational map $\theta:\widetilde S\to S$ with $\widetilde S\to X$ a projective morphism. There results a map of diagrams
$$\theta_\bullet:\widetilde S_\bullet\to S_\bullet$$
and associated functors
$$(\theta_\bullet)^\ast:\Coh(S_\bullet)\to\Coh(\widetilde S_\bullet)$$
and
$$(\theta_\bullet)_\ast:\Coh(\widetilde S_\bullet)\to\Coh(S_\bullet)$$

Given an object $V_\bullet\in\Coh(S_\bullet)$, there is a canoncal map
$$V_\bullet\to(\theta_\bullet)_\ast(\theta_\bullet)^\ast V_\bullet$$
that is an isomorphism over a dense open subspace of $S$. By Lemma \ref{lem:projective case described above}, there is thus a coherent sheaf $F$ on $\widetilde S$ and a map
$$s:V_\bullet\to(\theta_\ast F)_\bullet$$
that is an isomorphism over a dense open of $S$. The kernel and cokernel of $s$ are supported over proper closed subspaces of $S$. 

By Lemma~\ref{eimage-lemma}(\ref{cclosed}) and Lemma~\ref{eimage-lemma}(\ref{eclosed}), it suffices to show that $\ker(s)$ and $\coker(s)$ lie in the essential image of $\Coh(S)$. But this follows from the Noetherian induction hypothesis.
\end{proof}

\bibliographystyle{alpha}
\bibliography{citations}

\end{document}